\documentclass[12pt]{amsart}

\newtheorem{theorem}{Theorem}[section]
\newtheorem{proposition}[theorem]{Proposition}
\newtheorem{lemma}[theorem]{Lemma}

\theoremstyle{definition}

\newtheorem{question}[theorem]{Question}

\newtheorem{remark}[theorem]{Remark}

\def\cocoa{{\hbox{\rm C\kern-.13em o\kern-.07em C\kern-.13em o\kern-.15em A}}}

\begin{document}

\title{Balanced subdivisions and flips on surfaces}

\author{Satoshi Murai}
\address{
Satoshi Murai,
Department of Pure and Applied Mathematics,
Graduate School of Information Science and Technology,
Osaka University,
Suita, Osaka, 565-0871, Japan}
\email{s-murai@ist.osaka-u.ac.jp
}

\author{Yusuke Suzuki}
\address{
Department of Mathematics, Niigata University, 
8050 Ikarashi 2-no-cho, Nishi-ku, Niigata, 950-2181, Japan.
}
\email{y-suzuki@math.sc.niigata-u.ac.jp}


\begin{abstract}
In this paper, we show that two balanced triangulations of a closed surface 
are not necessary connected by a sequence of balanced stellar 
subdivisions and welds.
This answers a question posed by Izmestiev, Klee and Novik. 
We also show that two balanced triangulations of a closed surface 
are connected by a sequence of three local operations, 
which we call the pentagon 
contraction, the balanced edge subdivision and the balanced edge weld. 
In addition, we prove that two balanced triangulations of the $2$-sphere are 
connected by a sequence of pentagon contractions and their inverses 
if none of them are octahedral spheres. 
\end{abstract}

\maketitle

\section{Introduction}
It is a classical result in the combinatorial topology \cite{Al} which shows that two PL-homeomorphic simplicial complexes are connected by a sequence of stellar subdivisions and their inverses.
A closely related result is Pachner's result \cite{Pa1,Pa2} which shows that two PL-homeomorphic combinatorial manifolds are connected by a sequence of bistellar flips (see also \cite{Li} for the proofs of both results).
A combinatorial $d$-manifold is a triangulation of a $d$-manifold all whose vertex links are PL $(d-1)$-spheres. A combinatorial $d$-manifold is said to be {\em balanced} if its graph is $(d+1)$-colorable.
Recently, Izmestiev, Klee and Novik \cite{IKN} proved an analogue of Pachner's result for balanced combinatorial manifolds.
They introduced a version of bistellar flips that preserves the balanced property, which they call \textit{cross-flips}, and proved that two PL-homeomorphic balanced combinatorial manifolds are connected by a sequence of cross-flips.
In this paper, we study the following questions related to their result in the special case of triangulated surfaces.
\begin{itemize}
\item
There is an analogue of stellar subdivisions for balanced simplicial complexes, called {\em balanced stellar subdivisions} (see \cite[\S 2.5]{IKN}).
Are two PL-homeomorphic combinatorial manifolds connected by a sequence of balanced stellar subdivisions and their inverses?
\item
It is known that not all cross-flips are necessary to connect any two PL-homeomorphic balanced combinatorial manifolds. How many different types of cross-flips are indeed necessary?
\end{itemize}

A {\em triangulation} $G$ of a closed surface $F^2$ is a simple graph 
embedded on the surface such that each face of $G$ is bounded by a $3$-cycle and any two faces share at most one edge.
By a result of Izmestiev, Klee and Novik \cite[Theorem 1.1]{IKN}, two different balanced 
triangulations of a 
fixed closed surface are connected by a sequence of cross-flips.
A cross-flip in dimension $d$ is an operation that exchanges a 
shellable and co-shellable $d$-ball in the boundary of the cross 
$(d+1)$-polytope with its complement 
(see \cite{IKN} for the precise definition).
In dimension $2$, there are $9$ different types of cross-flips, but it is known that only $6$ flips, described in Figure~\ref{fig:6defo}, are necessary (see \cite[Remark 3.9]{IKN}).
Note that, in Figure~\ref{fig:6defo},
it is not allowed to make a double edge by the operations and each triangle must be a face.
In this paper, we call these six operations,
a balanced triangle subdivision (BT-subdivision or BTS), 
a balanced triangle weld (BT-weld or BTW), 
a balanced edge subdivision (BE-subdivision or BES), 
a balanced edge weld (BE-weld or BEW), 
a pentagon splitting (P-splitting or PS) and 
a pentagon contraction (P-contraction or PC). 
A BT-subdivision (resp., -weld) and 
a BE-subdivision (resp., -weld) are collectively referred to as 
{\em balanced subdivisions\/} (resp., -{\em welds\/}). 
Izmestiev, Klee and Novik \cite[Problem 3]{IKN} asked if balanced subdivisions and balanced welds suffice to transform any balanced triangulation of a closed surface into any other balanced triangulation of the same surface.
We answer this question.

\begin{figure}
\begin{center}
\input{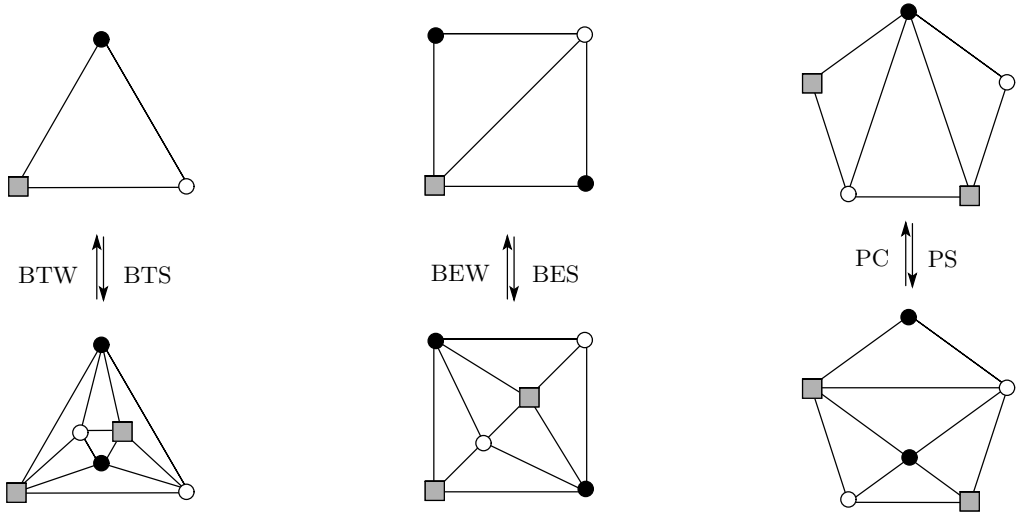}
\end{center}
\caption{Six necessary cross-flips.}
\label{fig:6defo}
\end{figure}

\begin{theorem}\label{thm1}
For every closed surface $F^2$, there are balanced 
triangulations $G$ and $G'$ of 
$F^2$ such that $G'$ cannot be obtained from $G$ by a 
sequence of balanced subdivisions and welds. 
\end{theorem}

Next, we consider how many different types of cross flips are necessary.
The above result shows that at least a P-contraction or a P-splitting is necessary.
Then since we can apply neither a P-contraction nor a P-splitting to the octahedral sphere (the boundary of the cross $3$-polytope),
we at least need three different cross-flips
to transform any balanced triangulation of the $2$-sphere to any other balanced triangulation of the $2$-sphere.
We show that a result proved by Kawarabayashi, Nakamoto and Suzuki in \cite{KNS} implies the following result which guarantees that three flips are indeed enough.

\begin{theorem}\label{thm:3operation}
Any two balanced triangulations of a closed surface $F^2$ are transformed into 
each other by a sequence of BE-subdivisions, BE-welds and P-contractions. 
\end{theorem}

As we mentioned, the set of three moves in the theorem is minimal possible.
However, somewhat surprisingly,
we show in Theorem \ref{finalthm} that most balanced triangulations of a fixed closed surface are actually connected by only P-splittings and P-contractions.
In particular, we prove the following strong statement for the $2$-sphere.

\begin{theorem}\label{thm:spheremain}
Any two balanced triangulations of the $2$-sphere except the octahedral sphere
can be transformed into each other by a sequence 
of P-splittings and P-contractions. 
\end{theorem}

This paper is organized as follows. 
In the next section, we introduce some 
operations defined for bipartite graphs, 
and show a key lemma to prove our main theorem. 
Section~\ref{sect:proof} 
is devoted to prove our first main result in 
the paper. 
In Section~\ref{sect:nece}, we discuss how many different types of cross-flips are sufficient to connect given two balanced triangulations of a closed surface.

\section{Operations for Bipartite graphs}\label{sect:bip}

In this section, we consider bipartite graphs 
which are not necessarily embedded on surfaces, 
and prove the key lemma to prove our first main theorem. 

We first introduce some notation.
In the paper, we consider simple graphs.
Let $G$ be a simple graph.
We denote by $V(G)$ the vertex set of $G$.
The \textit{degree} of the vertex $v$ in $G$ is the number of edges of $G$ that contains $v$. The \textit{minimal degree} of $G$ is the minimum of degrees of vertices of $G$.
An edge on vertices $a$ and $b$ will be denoted by $ab$ and a face on vertices $a,b$ and $c$ will be denoted by $abc$.
A graph $G$ is \textit{$d$-colorable} if there is a map $c:V(G) \to \{1,2,\dots,d\}$ such that $c(v) \ne c(u)$ for any edge $uv$ of $G$.
A $2$-colorable graph is called a \textit{bipartite graph}.
For bipartite graphs, 
we define the following three operations: 
Let $H$ be a bipartite graph. 
\begin{itemize}
\item[(I)] Add a pendant edge $vw$ with $v\in V(H)$ and $w\notin V(H)$. 
(A pendant edge is an edge such that one of its vertex has degree one.)
\item[(II)] Replace an edge $e=uv$ of $H$ with three edges $up,pq,qv$, where $p$ and $q$ are new vertices. 
\item[(III)] Add a vertex $w\notin V(H)$ and two incident edges $xw,wy$ where 
$x$ and $y$ have distance $2$ in $H$ (i.e., $xy$ is not an edge of $H$ and there is a vertex $z$ such that $xz$ and $yz$ are edges of $H$). 
\end{itemize}
The inverse operations of the above (I), (II) and (III) are represented 
by (I'), (II') and (III'), respectively (see Figure~\ref{fig:6operations}). 
In particular, we call (II) the \textit{subdivision} of $uv$ and call (II') the \textit{smoothing} of the edges $up,pq,qv$.
Note that each of these six operations preserves the bipartiteness of the graph.

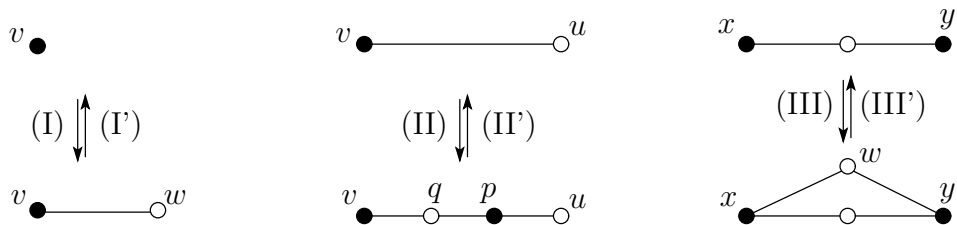
\begin{figure}[tb]
\begin{center}
{\unitlength 0.1in%
\begin{picture}( 49.3000, 11.5000)(  8.7000,-15.6000)%
%
\special{sh 1.000}%
\special{ia 1020 630 40 40  0.0000000  6.2831853}%
\special{pn 8}%
\special{ar 1020 630 40 40  0.0000000  6.2831853}%
%
\special{sh 1.000}%
\special{ia 1020 1490 40 40  0.0000000  6.2831853}%
\special{pn 8}%
\special{ar 1020 1490 40 40  0.0000000  6.2831853}%
%
\special{pn 8}%
\special{pa 1020 1500}%
\special{pa 1660 1500}%
\special{fp}%
%
\special{sh 0}%
\special{ia 1650 1490 40 40  0.0000000  6.2831853}%
\special{pn 8}%
\special{ar 1650 1490 40 40  0.0000000  6.2831853}%
%
\special{sh 1.000}%
\special{ia 2730 620 40 40  0.0000000  6.2831853}%
\special{pn 8}%
\special{ar 2730 620 40 40  0.0000000  6.2831853}%
%
\special{pn 8}%
\special{pa 2730 620}%
\special{pa 3750 620}%
\special{fp}%
%
\special{sh 0}%
\special{ia 3760 620 40 40  0.0000000  6.2831853}%
\special{pn 8}%
\special{ar 3760 620 40 40  0.0000000  6.2831853}%
%
\special{sh 1.000}%
\special{ia 2730 1520 40 40  0.0000000  6.2831853}%
\special{pn 8}%
\special{ar 2730 1520 40 40  0.0000000  6.2831853}%
%
\special{pn 8}%
\special{pa 2730 1520}%
\special{pa 3750 1520}%
\special{fp}%
%
\special{sh 0}%
\special{ia 3760 1520 40 40  0.0000000  6.2831853}%
\special{pn 8}%
\special{ar 3760 1520 40 40  0.0000000  6.2831853}%
%
\special{sh 0}%
\special{ia 3080 1520 40 40  0.0000000  6.2831853}%
\special{pn 8}%
\special{ar 3080 1520 40 40  0.0000000  6.2831853}%
%
\special{sh 1.000}%
\special{ia 3410 1520 40 40  0.0000000  6.2831853}%
\special{pn 8}%
\special{ar 3410 1520 40 40  0.0000000  6.2831853}%
%
\special{sh 1.000}%
\special{ia 4730 620 40 40  0.0000000  6.2831853}%
\special{pn 8}%
\special{ar 4730 620 40 40  0.0000000  6.2831853}%
%
\special{pn 8}%
\special{pa 4730 620}%
\special{pa 5750 620}%
\special{fp}%
%
\special{sh 1.000}%
\special{ia 5760 620 40 40  0.0000000  6.2831853}%
\special{pn 8}%
\special{ar 5760 620 40 40  0.0000000  6.2831853}%
%
\special{sh 0}%
\special{ia 5260 620 40 40  0.0000000  6.2831853}%
\special{pn 8}%
\special{ar 5260 620 40 40  0.0000000  6.2831853}%
%
\special{sh 1.000}%
\special{ia 4730 1520 40 40  0.0000000  6.2831853}%
\special{pn 8}%
\special{ar 4730 1520 40 40  0.0000000  6.2831853}%
%
\special{pn 8}%
\special{pa 4730 1520}%
\special{pa 5750 1520}%
\special{fp}%
%
\special{sh 1.000}%
\special{ia 5760 1520 40 40  0.0000000  6.2831853}%
\special{pn 8}%
\special{ar 5760 1520 40 40  0.0000000  6.2831853}%
%
\special{sh 0}%
\special{ia 5260 1520 40 40  0.0000000  6.2831853}%
\special{pn 8}%
\special{ar 5260 1520 40 40  0.0000000  6.2831853}%
%
\special{pn 8}%
\special{pa 4730 1520}%
\special{pa 5260 1270}%
\special{fp}%
\special{pa 5260 1270}%
\special{pa 5760 1520}%
\special{fp}%
%
\special{sh 0}%
\special{ia 5260 1260 40 40  0.0000000  6.2831853}%
\special{pn 8}%
\special{ar 5260 1260 40 40  0.0000000  6.2831853}%
%
\special{pn 8}%
\special{pa 1230 910}%
\special{pa 1230 1230}%
\special{fp}%
\special{sh 1}%
\special{pa 1230 1230}%
\special{pa 1250 1163}%
\special{pa 1230 1177}%
\special{pa 1210 1163}%
\special{pa 1230 1230}%
\special{fp}%
%
\special{pn 8}%
\special{pa 1270 1210}%
\special{pa 1270 890}%
\special{fp}%
\special{sh 1}%
\special{pa 1270 890}%
\special{pa 1250 957}%
\special{pa 1270 943}%
\special{pa 1290 957}%
\special{pa 1270 890}%
\special{fp}%
%
\special{pn 8}%
\special{pa 3230 910}%
\special{pa 3230 1230}%
\special{fp}%
\special{sh 1}%
\special{pa 3230 1230}%
\special{pa 3250 1163}%
\special{pa 3230 1177}%
\special{pa 3210 1163}%
\special{pa 3230 1230}%
\special{fp}%
%
\special{pn 8}%
\special{pa 3270 1210}%
\special{pa 3270 890}%
\special{fp}%
\special{sh 1}%
\special{pa 3270 890}%
\special{pa 3250 957}%
\special{pa 3270 943}%
\special{pa 3290 957}%
\special{pa 3270 890}%
\special{fp}%
%
\special{pn 8}%
\special{pa 5237 810}%
\special{pa 5237 1130}%
\special{fp}%
\special{sh 1}%
\special{pa 5237 1130}%
\special{pa 5257 1063}%
\special{pa 5237 1077}%
\special{pa 5217 1063}%
\special{pa 5237 1130}%
\special{fp}%
%
\special{pn 8}%
\special{pa 5277 1110}%
\special{pa 5277 790}%
\special{fp}%
\special{sh 1}%
\special{pa 5277 790}%
\special{pa 5257 857}%
\special{pa 5277 843}%
\special{pa 5297 857}%
\special{pa 5277 790}%
\special{fp}%
\put(9.8000,-11.3000){\makebox(0,0)[lb]{(I)}}%
\put(29.2000,-11.3000){\makebox(0,0)[lb]{(II)}}%
\put(48.8000,-10.3000){\makebox(0,0)[lb]{(III)}}%
\put(13.4000,-11.3000){\makebox(0,0)[lb]{(I')}}%
\put(33.4000,-11.3000){\makebox(0,0)[lb]{(II')}}%
\put(53.2000,-10.3000){\makebox(0,0)[lb]{(III')}}%
\put(8.7000,-6.1000){\makebox(0,0)[lb]{$v$}}%
\put(8.8000,-14.6000){\makebox(0,0)[lb]{$v$}}%
\put(16.8000,-14.6000){\makebox(0,0)[lb]{$w$}}%
\put(25.8000,-6.1000){\makebox(0,0)[lb]{$v$}}%
\put(38.0000,-5.7000){\makebox(0,0)[lb]{$u$}}%
\put(38.0000,-14.8000){\makebox(0,0)[lb]{$u$}}%
\put(26.1000,-14.6000){\makebox(0,0)[lb]{$v$}}%
\put(30.6000,-14.5000){\makebox(0,0)[lb]{$q$}}%
\put(33.4000,-14.5000){\makebox(0,0)[lb]{$p$}}%
\put(45.8000,-5.6000){\makebox(0,0)[lb]{$x$}}%
\put(57.4000,-5.4000){\makebox(0,0)[lb]{$y$}}%
\put(45.9000,-14.7000){\makebox(0,0)[lb]{$x$}}%
\put(57.4000,-14.5000){\makebox(0,0)[lb]{$y$}}%
\put(53.2000,-12.4000){\makebox(0,0)[lb]{$w$}}%
\end{picture}}%
\end{center}
\caption{Six operations defined for bipartite graphs.}
\label{fig:6operations}
\end{figure}

A set of two adjacent vertices $\{p,q\}$ 
of degree $2$ in a bipartite graph $H$ 
is said to be {\em smoothable} if it is possible to apply (II') that removes the vertices $p$ and $q$ to $H$; that is, 
there exists no cycle of length $4$ containing $p$ and $q$. 
Furthermore, a vertex $w$ of degree $2$ in a bipartite 
graph $H$ is said to be {\em removable\/} 
if we can remove the vertex $w$ by applying (III'); that is, there exists a $4$-cycle in $H$ 
containing $w$. 
The following lemma plays an important role when we prove 
our main theorem in the next section. 

\begin{lemma}\label{lm:key}
Let $H$ be a bipartite graph with minimum degree at least $3$. 
If $H'$ is obtained from $H$ by a sequence of operations (I), (II), (III), 
(I'), (II') and (III'), 
then $H'$ is obtained from $H$ by a sequence of operations (I), (II) and (III). 
\end{lemma}

\begin{proof}
In the following argument, 
we say that a bipartite graph is {\em configurable\/} from $H$ (by at most $t$ steps) if it can be obtained from $H$ by applying operations (I), (II) and (III) (at most $t$ times).

Let $H'$ be a graph obtained from $H$ by a sequence 
of operations (I), (II), (III), (I'), (II') and (III'). 
Then, there is a sequence of bipartite graphs 
$H=H_0, H_1, \ldots, H_t=H'$ such that 
$H_{i+1}$ is obtained from $H_i$ by one of the 
six operations for $i=0, \ldots, t-1$, as shown in the 
following diagram. 
\[
H=H_0
\stackrel{o_1} \longrightarrow H_1
\stackrel{o_2} \longrightarrow  H_2
\longrightarrow \cdots 
\stackrel{o_{t-2}} \longrightarrow H_{t-2}
\stackrel{o_{t-1}} \longrightarrow H_{t-1}
\stackrel{o_t} \longrightarrow H_{t}=H'.
\]
We claim that $H_t=H'$ is configurable from $H$ by at most $t$ steps.
We proceed by induction on $t$.
Since any vertex of $H$ has degree at least $3$,
$o_1$ must be (I), (II) or (III).
Thus the assertion is obvious when $t=1$.
Suppose $t \geq 2$.
To prove the desired assertion, it only suffices to show the case when each of $o_1,\dots,o_{t-1}$ is one of (I), (II) and (III), 
and $o_t$ is one of (I'), (II') and (III').

\bigskip
\noindent[Case 1] Suppose that $o_t$ is (I') which
removes a vertex $w$ and an edge $vw$ from $H_{t-1}$. 
If $w$ is not a vertex of $H_{t-2}$, then $o_{t-1}$ should be 
(I) which add $w$ and $vw$; note that each of (II) and (III) 
does not generate a new vertex of degree $1$. 
In this case, it is clear that $H_{t-2}=H_t$ and 
hence $H_t$ is configurable. 
Thus, we assume that $w$ is a vertex of $H_{t-2}$. 
Since none of (I), (II) and (III) decrease the degrees of vertices, 
$w$ has degree $1$ in $H_{t-2}$. 
Let $v'$ denotes the unique neighbor of $w$ in $H_{t-2}$. 

First, suppose that $v \ne v'$. 
In this case, $o_{t-1}$ should be (II) that subdivide $v'w$, 
and hence a graph isomorphic to $H_t$ can be obtained from $H_{t-2}$ by adding a pendant edge to $w$ by applying (I) (see Figure~\ref{fig:case11}).
Next, we suppose that $v=v'$. 
We delete a vertex $w$ from $H_{t-2}$ and denote the resulting graph by 
$H'_{t-2}$.
Since $H_{t-2} \to H_{t-2}'$ is an operation (I'),  
by the induction hypothesis, $H'_{t-2}$ is configurable from $H$ by at most $t-1$ steps.
Furthermore, since $o_{t-1}$ is not (II) which subdivides $wv$, 
we can apply the same operation as $o_{t-1}$ 
to $H'_{t-2}$ and obtain a graph isomorphic to $H_t$. 
Therefore, $H_t$ is configurable from $H$ also in this case. 

\begin{figure}[tb]
\begin{center}
{\unitlength 0.1in%
\begin{picture}( 44.6000, 10.4000)(  7.9000,-15.0000)%
%
\special{sh 1.000}%
\special{ia 4960 1250 40 40  0.0000000  6.2831853}%
\special{pn 8}%
\special{ar 4960 1250 40 40  0.0000000  6.2831853}%
%
\special{pn 8}%
\special{pa 4950 1250}%
\special{pa 5210 830}%
\special{fp}%
%
\special{sh 1.000}%
\special{ia 960 1250 40 40  0.0000000  6.2831853}%
\special{pn 8}%
\special{ar 960 1250 40 40  0.0000000  6.2831853}%
\put(8.8000,-11.9000){\makebox(0,0)[lb]{$v'$}}%
%
\special{pn 8}%
\special{pa 960 1250}%
\special{pa 1330 610}%
\special{fp}%
%
\special{pn 8}%
\special{pa 790 1250}%
\special{pa 950 1250}%
\special{fp}%
\special{pa 950 1250}%
\special{pa 950 1440}%
\special{fp}%
%
\special{sh 0}%
\special{ia 1310 650 40 40  0.0000000  6.2831853}%
\special{pn 8}%
\special{ar 1310 650 40 40  0.0000000  6.2831853}%
%
\special{pn 8}%
\special{pa 1790 1010}%
\special{pa 2340 1010}%
\special{fp}%
\special{sh 1}%
\special{pa 2340 1010}%
\special{pa 2273 990}%
\special{pa 2287 1010}%
\special{pa 2273 1030}%
\special{pa 2340 1010}%
\special{fp}%
\put(11.7000,-5.9000){\makebox(0,0)[lb]{$w$}}%
\put(11.1000,-16.3000){\makebox(0,0)[lb]{$H_{t-2}$}}%
\put(19.7000,-11.8000){\makebox(0,0)[lb]{$o_{t-1}$}}%
%
\special{sh 1.000}%
\special{ia 2960 1250 40 40  0.0000000  6.2831853}%
\special{pn 8}%
\special{ar 2960 1250 40 40  0.0000000  6.2831853}%
\put(30.4000,-8.4000){\makebox(0,0)[lb]{$v$}}%
%
\special{pn 8}%
\special{pa 2960 1250}%
\special{pa 3330 610}%
\special{fp}%
%
\special{pn 8}%
\special{pa 2790 1250}%
\special{pa 2950 1250}%
\special{fp}%
\special{pa 2950 1250}%
\special{pa 2950 1440}%
\special{fp}%
%
\special{sh 0}%
\special{ia 3310 650 40 40  0.0000000  6.2831853}%
\special{pn 8}%
\special{ar 3310 650 40 40  0.0000000  6.2831853}%
%
\special{pn 8}%
\special{pa 3790 1010}%
\special{pa 4340 1010}%
\special{fp}%
\special{sh 1}%
\special{pa 4340 1010}%
\special{pa 4273 990}%
\special{pa 4287 1010}%
\special{pa 4273 1030}%
\special{pa 4340 1010}%
\special{fp}%
\put(31.7000,-5.9000){\makebox(0,0)[lb]{$w$}}%
\put(31.1000,-16.3000){\makebox(0,0)[lb]{$H_{t-1}$}}%
\put(39.7000,-11.8000){\makebox(0,0)[lb]{$o_{t}$}}%
%
\special{sh 1.000}%
\special{ia 3200 840 40 40  0.0000000  6.2831853}%
\special{pn 8}%
\special{ar 3200 840 40 40  0.0000000  6.2831853}%
\put(50.4000,-8.4000){\makebox(0,0)[lb]{$v$}}%
%
\special{pn 8}%
\special{pa 4790 1250}%
\special{pa 4950 1250}%
\special{fp}%
\special{pa 4950 1250}%
\special{pa 4950 1440}%
\special{fp}%
\put(49.9000,-16.3000){\makebox(0,0)[lb]{$H_t$}}%
%
\special{sh 0}%
\special{ia 5080 1040 40 40  0.0000000  6.2831853}%
\special{pn 8}%
\special{ar 5080 1040 40 40  0.0000000  6.2831853}%
%
\special{sh 1.000}%
\special{ia 5210 830 40 40  0.0000000  6.2831853}%
\special{pn 8}%
\special{ar 5210 830 40 40  0.0000000  6.2831853}%
%
\special{sh 0}%
\special{ia 3080 1040 40 40  0.0000000  6.2831853}%
\special{pn 8}%
\special{ar 3080 1040 40 40  0.0000000  6.2831853}%
\put(28.8000,-11.9000){\makebox(0,0)[lb]{$v'$}}%
\put(48.8000,-11.9000){\makebox(0,0)[lb]{$v'$}}%
\end{picture}}%
\end{center}
\caption{Configurations in Case 1 in the proof of Lemma~\ref{lm:key}.}
\label{fig:case11}
\end{figure}


\bigskip
\noindent[Case 2] Suppose that $o_t$ is (II') that replace edges 
$up,pq,qv$ with $uv$. 
First, suppose that both $p$ and $q$ are vertices of $H_{t-2}$
and $\{p,q\}$ is smoothable in $H_{t-2}$. 
Let $u'$ and $v'$ denote the vertices such that $u'p,pq,qv'$ are edges of $H_{t-2}$.
We apply (II') that replace $u'p,pq,qv'$ with 
$u'v'$ to $H_{t-2}$ 
and denote the resulting graph by $H'_{t-2}$. 
By the induction hypothesis, $H'_{t-2}$ is configurable from $H$ by at most $t-1$ steps. 
If $o_{t-1}$ is not (II) which subdivides either $u'p$ or $qv'$, 
then we can apply $o_{t-1}$ to $H'_{t-2}$ and obtain a 
graph isomorphic to $H_t$. 
On the other hand, 
if $o_{t-1}$ is (II) that subdivides either $u'p$ or $qv'$, 
then $H_t$ and $H_{t-2}$ are clearly isomorphic. 
In either case, $H_t$ is configurable from $H$ by at most $t$ steps.

By the above argument, 
we only need to discuss the case when at least one of $p$ and $q$ is not a vertex of $H_{t-2}$ or  $\{p, q\}$ is not smoothable in $H_{t-2}$. 
We divide the argument into three cases (A), (B) and (C) depending 
on the situation.

(A) Neither $p$ nor $q$ is a vertex of $H_{t-2}$:  
In this case, $o_{t-1}$ is clearly an operation adding 
$p$ and $q$, that is, $o_{t-1}$ is (II) that subdivides $uv$ in $H_{t-2}$. 
It is easy to see that $H_{t-2}=H_t$.

(B) $p$ is a vertex of $H_{t-2}$ but $q$ is not of $H_{t-2}$: 
Note that there exists no cycle of length $4$ containing 
$p$ and $q$ in $H_{t-1}$ since $\{p, q\}$ is smoothable in 
$H_{t-1}$. 
Under the condition, $q$ must be added by $o_{t-1}$, and 
we can conclude that $o_{t-1}$ is (II) that subdivide an edge incident to 
$p$.
(If $o_{t-1}$ is (III), then $p$ and $q$ would lie on a $4$-cycle 
in $H_{t-1}$.) 
As a result, $H_{t-2}$ is isomorphic to $H_t$ and hence 
$H_t$ is configurable from $H$.

\begin{figure}[tb]
\begin{center}
\input{Fig6-2.tex}
\end{center}
\caption{Configurations of (C) in Case 2 in the proof of Lemma~\ref{lm:key}.}
\label{fig:case2}
\end{figure}

(C) Both of $p$ and $q$ are the vertices of $H_{t-2}$: 
Here note that $p$ and $q$ are adjacent and have degree at most $2$ 
in $H_{t-2}$ since each of (I), (II) and (III) does not decrease 
the degrees of vertices and does not join two non-adjacent vertices. 
If one of $p$ and $q$, say $q$, has degree $1$, then $o_{t-1}$ should be (I) 
that add an edge incident to $q$ since (III) would generate a $4$-cycle containing $p$ and $q$. 
In this case, a graph isomorphic to $H_t$ can be obtained from $H_{t-2}$ 
by deleting $q$ using operation (I'), and hence $H_t$ is configurable from $H$ 
by the induction 
hypothesis (see the upper diagram in Figure~\ref{fig:case2}). 
On the other hand, if each of $p$ and $q$ has degree $2$, 
then there should exist a $4$-cycle containing $p$ and $q$ in $H_{t-2}$ 
under our assumption. 
Since $\{p, q\}$ is smoothable in $H_{t-1}$,
$o_{t-1}$ should be (II) that subdivide an edge on the $4$-cycle. 
In any case, $H_{t-2}$ and $H_{t}$ is isomorphic to each other 
(see the bottom diagram in Figure~\ref{fig:case2}).

\bigskip

\noindent
[Case 3]
Suppose that $o_t$ is (III') deleting a vertex $w$ of degree $2$ 
and two edges $xw$ and $yw$. 
Note that $H_{t-1}$ must have a $4$-cycle that contains $w$.
First assume that $w$ is a vertex of $H_{t-2}$ 
and is removable in $H_{t-2}$. 
Let $x'$ and $y'$ denote the vertices adjacent to $w$ in $H_{t-2}$. 
Now, since there exists a $4$-cycle containing $w$ in $H_{t-1}$, 
$o_{t-1}$ is not (II) that subdivides $x'w$ or $wy'$. 
Thus, we have $\{x',y'\}=\{x,y\}$. 
We delete $w$ from $H_{t-2}$ by applying (III') and denote the resulting graph 
by $H'_{t-2}$. 
Since $o_{t-1}$ is not (II) subdividing $xw$ or $wy$, 
we can apply the same operation as $o_{t-1}$ to $H'_{t-2}$ and 
obtain a graph isomorphic to $H_t$. 
By the induction hypothesis, $H_t$ is configurable from $H$ by at most $t$ steps.

By the above argument, 
we may assume that $w$ is not a vertex of $H_{t-2}$ or 
$w$ is not removable in $H_{t-2}$. 
It suffices to discuss the following three cases (A), (B) and (C). 

(A) $w$ is not a vertex of $H_{t-2}$: 
Clearly, $w$ must be added by $o_{t-1}$. 
Since $H_{t-1}$ has a $4$-cycle containing $w$, 
$o_{t-1}$ cannot be (II); that is, $o_{t-1}$ should be (III). 
Then, it is easy to see that $H_{t-2}=H_t$. 

(B) $w$ has degree $1$ in $H_{t-2}$: 
In this case, $o_{t-1}$ is clearly (III). 
We assume that $o_{t-1}$ adds a vertex $v$ and edges $wv$ 
and $vu$ (see Figure~\ref{fig:case3}). 
We remove $w$ from $H_{t-2}$ by applying (I'), 
and denote the resulting graph by $H_{t-2}'$. 
By the induction hypothesis, $H_{t-2}'$ is configurable from $H$ by at most $t-1$ steps.
Furthermore, $H_t$ is obtained from 
$H_{t-2}'$ by (I) which adds an edge incident to $u$. 
Thus, $H_t$ is also configurable from $H$ by at most $t$ steps.

\begin{figure}[tb]
\begin{center}
{\unitlength 0.1in%
\begin{picture}( 47.2000,  9.3000)(  7.8000,-15.0000)%
%
\special{pn 8}%
\special{pa 1620 1250}%
\special{pa 1460 1250}%
\special{fp}%
\special{pa 1460 1250}%
\special{pa 1460 1440}%
\special{fp}%
%
\special{pn 8}%
\special{pa 1790 1010}%
\special{pa 2340 1010}%
\special{fp}%
\special{sh 1}%
\special{pa 2340 1010}%
\special{pa 2273 990}%
\special{pa 2287 1010}%
\special{pa 2273 1030}%
\special{pa 2340 1010}%
\special{fp}%
\put(11.1000,-16.3000){\makebox(0,0)[lb]{$H_{t-2}$}}%
\put(19.7000,-11.8000){\makebox(0,0)[lb]{$o_{t-1}$}}%
%
\special{pn 8}%
\special{pa 3790 1010}%
\special{pa 4340 1010}%
\special{fp}%
\special{sh 1}%
\special{pa 4340 1010}%
\special{pa 4273 990}%
\special{pa 4287 1010}%
\special{pa 4273 1030}%
\special{pa 4340 1010}%
\special{fp}%
\put(31.1000,-16.3000){\makebox(0,0)[lb]{$H_{t-1}$}}%
\put(39.7000,-11.8000){\makebox(0,0)[lb]{$o_{t}$}}%
\put(49.9000,-16.3000){\makebox(0,0)[lb]{$H_t$}}%
%
\special{sh 1.000}%
\special{ia 950 1250 40 40  0.0000000  6.2831853}%
\special{pn 8}%
\special{ar 950 1250 40 40  0.0000000  6.2831853}%
%
\special{pn 8}%
\special{pa 780 1250}%
\special{pa 940 1250}%
\special{fp}%
\special{pa 940 1250}%
\special{pa 940 1440}%
\special{fp}%
\put(8.9000,-7.0000){\makebox(0,0)[lb]{$w$}}%
%
\special{pn 8}%
\special{pa 950 1240}%
\special{pa 950 760}%
\special{fp}%
%
\special{sh 0}%
\special{ia 950 770 40 40  0.0000000  6.2831853}%
\special{pn 8}%
\special{ar 950 770 40 40  0.0000000  6.2831853}%
\put(15.0000,-12.1000){\makebox(0,0)[lb]{$u$}}%
%
\special{pn 8}%
\special{pa 1420 1250}%
\special{pa 940 1250}%
\special{fp}%
%
\special{sh 0}%
\special{ia 1450 1250 40 40  0.0000000  6.2831853}%
\special{pn 8}%
\special{ar 1450 1250 40 40  0.0000000  6.2831853}%
%
\special{pn 8}%
\special{pa 3510 1250}%
\special{pa 3350 1250}%
\special{fp}%
\special{pa 3350 1250}%
\special{pa 3350 1440}%
\special{fp}%
%
\special{sh 1.000}%
\special{ia 2840 1250 40 40  0.0000000  6.2831853}%
\special{pn 8}%
\special{ar 2840 1250 40 40  0.0000000  6.2831853}%
%
\special{pn 8}%
\special{pa 2670 1250}%
\special{pa 2830 1250}%
\special{fp}%
\special{pa 2830 1250}%
\special{pa 2830 1440}%
\special{fp}%
\put(27.8000,-7.0000){\makebox(0,0)[lb]{$w$}}%
%
\special{pn 8}%
\special{pa 2840 1240}%
\special{pa 2840 760}%
\special{fp}%
%
\special{pn 8}%
\special{pa 3310 770}%
\special{pa 2830 770}%
\special{fp}%
%
\special{sh 0}%
\special{ia 2840 770 40 40  0.0000000  6.2831853}%
\special{pn 8}%
\special{ar 2840 770 40 40  0.0000000  6.2831853}%
%
\special{sh 1.000}%
\special{ia 3340 770 40 40  0.0000000  6.2831853}%
\special{pn 8}%
\special{ar 3340 770 40 40  0.0000000  6.2831853}%
\put(32.5000,-7.0000){\makebox(0,0)[lb]{$v$}}%
%
\special{pn 8}%
\special{pa 3340 1240}%
\special{pa 3340 760}%
\special{fp}%
%
\special{pn 8}%
\special{pa 3310 1250}%
\special{pa 2830 1250}%
\special{fp}%
%
\special{sh 0}%
\special{ia 3340 1250 40 40  0.0000000  6.2831853}%
\special{pn 8}%
\special{ar 3340 1250 40 40  0.0000000  6.2831853}%
%
\special{pn 8}%
\special{pa 5500 1250}%
\special{pa 5340 1250}%
\special{fp}%
\special{pa 5340 1250}%
\special{pa 5340 1440}%
\special{fp}%
%
\special{sh 1.000}%
\special{ia 4830 1250 40 40  0.0000000  6.2831853}%
\special{pn 8}%
\special{ar 4830 1250 40 40  0.0000000  6.2831853}%
%
\special{pn 8}%
\special{pa 4660 1250}%
\special{pa 4820 1250}%
\special{fp}%
\special{pa 4820 1250}%
\special{pa 4820 1440}%
\special{fp}%
%
\special{sh 1.000}%
\special{ia 5330 770 40 40  0.0000000  6.2831853}%
\special{pn 8}%
\special{ar 5330 770 40 40  0.0000000  6.2831853}%
\put(52.4000,-7.0000){\makebox(0,0)[lb]{$v$}}%
%
\special{pn 8}%
\special{pa 5330 1240}%
\special{pa 5330 760}%
\special{fp}%
%
\special{pn 8}%
\special{pa 5300 1250}%
\special{pa 4820 1250}%
\special{fp}%
%
\special{sh 0}%
\special{ia 5330 1250 40 40  0.0000000  6.2831853}%
\special{pn 8}%
\special{ar 5330 1250 40 40  0.0000000  6.2831853}%
\put(34.0000,-12.1000){\makebox(0,0)[lb]{$u$}}%
\put(53.9000,-12.1000){\makebox(0,0)[lb]{$u$}}%
\end{picture}}%
\end{center}
\caption{Configurations of (B) in Case 3 in the proof of Lemma~\ref{lm:key}.}
\label{fig:case3}
\end{figure}

\begin{figure}[tb]
\begin{center}
{\unitlength 0.1in%
\begin{picture}( 47.1000,  9.3000)(  7.8000,-15.0000)%
%
\special{pn 8}%
\special{pa 1420 770}%
\special{pa 940 770}%
\special{fp}%
%
\special{pn 8}%
\special{pa 1620 770}%
\special{pa 1460 770}%
\special{fp}%
\special{pa 1460 770}%
\special{pa 1460 580}%
\special{fp}%
%
\special{pn 8}%
\special{pa 1790 1010}%
\special{pa 2340 1010}%
\special{fp}%
\special{sh 1}%
\special{pa 2340 1010}%
\special{pa 2273 990}%
\special{pa 2287 1010}%
\special{pa 2273 1030}%
\special{pa 2340 1010}%
\special{fp}%
\put(11.1000,-16.3000){\makebox(0,0)[lb]{$H_{t-2}$}}%
\put(19.7000,-11.8000){\makebox(0,0)[lb]{$o_{t-1}$}}%
%
\special{pn 8}%
\special{pa 3790 1010}%
\special{pa 4340 1010}%
\special{fp}%
\special{sh 1}%
\special{pa 4340 1010}%
\special{pa 4273 990}%
\special{pa 4287 1010}%
\special{pa 4273 1030}%
\special{pa 4340 1010}%
\special{fp}%
\put(31.1000,-16.3000){\makebox(0,0)[lb]{$H_{t-1}$}}%
\put(39.7000,-11.8000){\makebox(0,0)[lb]{$o_{t}$}}%
\put(49.9000,-16.3000){\makebox(0,0)[lb]{$H_t$}}%
%
\special{sh 1.000}%
\special{ia 950 1250 40 40  0.0000000  6.2831853}%
\special{pn 8}%
\special{ar 950 1250 40 40  0.0000000  6.2831853}%
%
\special{pn 8}%
\special{pa 780 1250}%
\special{pa 940 1250}%
\special{fp}%
\special{pa 940 1250}%
\special{pa 940 1440}%
\special{fp}%
\put(8.9000,-7.0000){\makebox(0,0)[lb]{$w$}}%
%
\special{pn 8}%
\special{pa 950 1240}%
\special{pa 950 760}%
\special{fp}%
%
\special{sh 0}%
\special{ia 950 770 40 40  0.0000000  6.2831853}%
\special{pn 8}%
\special{ar 950 770 40 40  0.0000000  6.2831853}%
\put(15.0000,-7.3000){\makebox(0,0)[lb]{$y'$}}%
%
\special{sh 1.000}%
\special{ia 2840 1250 40 40  0.0000000  6.2831853}%
\special{pn 8}%
\special{ar 2840 1250 40 40  0.0000000  6.2831853}%
%
\special{pn 8}%
\special{pa 2670 1250}%
\special{pa 2830 1250}%
\special{fp}%
\special{pa 2830 1250}%
\special{pa 2830 1440}%
\special{fp}%
\put(27.8000,-7.0000){\makebox(0,0)[lb]{$w$}}%
%
\special{pn 8}%
\special{pa 2840 1240}%
\special{pa 2840 760}%
\special{fp}%
%
\special{pn 8}%
\special{pa 3310 770}%
\special{pa 2830 770}%
\special{fp}%
%
\special{sh 0}%
\special{ia 2840 770 40 40  0.0000000  6.2831853}%
\special{pn 8}%
\special{ar 2840 770 40 40  0.0000000  6.2831853}%
%
\special{sh 1.000}%
\special{ia 3340 770 40 40  0.0000000  6.2831853}%
\special{pn 8}%
\special{ar 3340 770 40 40  0.0000000  6.2831853}%
%
\special{pn 8}%
\special{pa 3340 1240}%
\special{pa 3340 760}%
\special{fp}%
%
\special{pn 8}%
\special{pa 3310 1250}%
\special{pa 2830 1250}%
\special{fp}%
%
\special{sh 0}%
\special{ia 3340 1250 40 40  0.0000000  6.2831853}%
\special{pn 8}%
\special{ar 3340 1250 40 40  0.0000000  6.2831853}%
%
\special{sh 1.000}%
\special{ia 4830 1250 40 40  0.0000000  6.2831853}%
\special{pn 8}%
\special{ar 4830 1250 40 40  0.0000000  6.2831853}%
%
\special{pn 8}%
\special{pa 4660 1250}%
\special{pa 4820 1250}%
\special{fp}%
\special{pa 4820 1250}%
\special{pa 4820 1440}%
\special{fp}%
%
\special{sh 1.000}%
\special{ia 5330 770 40 40  0.0000000  6.2831853}%
\special{pn 8}%
\special{ar 5330 770 40 40  0.0000000  6.2831853}%
%
\special{pn 8}%
\special{pa 5330 1240}%
\special{pa 5330 760}%
\special{fp}%
%
\special{pn 8}%
\special{pa 5300 1250}%
\special{pa 4820 1250}%
\special{fp}%
%
\special{sh 0}%
\special{ia 5330 1250 40 40  0.0000000  6.2831853}%
\special{pn 8}%
\special{ar 5330 1250 40 40  0.0000000  6.2831853}%
\put(34.0000,-12.1000){\makebox(0,0)[lb]{$v$}}%
\put(53.9000,-12.1000){\makebox(0,0)[lb]{$v$}}%
%
\special{sh 1.000}%
\special{ia 1460 770 40 40  0.0000000  6.2831853}%
\special{pn 8}%
\special{ar 1460 770 40 40  0.0000000  6.2831853}%
\put(7.9000,-14.2000){\makebox(0,0)[lb]{$x'$}}%
%
\special{pn 8}%
\special{pa 3500 770}%
\special{pa 3340 770}%
\special{fp}%
\special{pa 3340 770}%
\special{pa 3340 580}%
\special{fp}%
\put(33.8000,-7.3000){\makebox(0,0)[lb]{$y'$}}%
%
\special{pn 8}%
\special{pa 5490 770}%
\special{pa 5330 770}%
\special{fp}%
\special{pa 5330 770}%
\special{pa 5330 580}%
\special{fp}%
\put(53.7000,-7.3000){\makebox(0,0)[lb]{$y'$}}%
\put(26.7000,-14.2000){\makebox(0,0)[lb]{$x'$}}%
\put(46.6000,-14.2000){\makebox(0,0)[lb]{$x'$}}%
\end{picture}}%
\end{center}
\caption{Configurations of (C) in Case 3 in the 
proof of Lemma~\ref{lm:key}.}
\label{fig:case3(C)}
\end{figure}

(C) $w$ is a vertex of degree $2$ in $H_{t-2}$: 
Denote two vertices adjacent to $w$ in $H_{t-2}$ by $x'$ and $y'$. 
By our assumption, $w$ is not removable in $H_{t-2}$, that is, 
there exists no cycle of length $4$ containing $w$. 
On the other hand, $w$ is removable and there exists such a 
$4$-cycle in $H_{t-1}$. 
To satisfy these conditions, $o_{t-1}$ should be (III) 
which adds a vertex $v$ and two edges $x'v$ and $vy'$. 
However, it is easy to see that $H_{t-2}$ is isomorphic to $H_t$
(see Figure~\ref{fig:case3(C)}).

Now, we have considered all cases and hence the 
lemma follows. 
\end{proof}

\section{Proof of Theorem~\ref{thm1}}\label{sect:proof}

An even embedding $H$ of a closed surface $F^2$
is a graph embedded on $F^2$ such that each face of $H$ is bounded by a cycle of even length.
For an even embedding $H$ of $F^2$, its \textit{face subdivision}, denoted by $S(H)$, is the triangulation of $F^2$ obtained from $H$ by adding a new vertex into each face of $H$ and joining it all vertices on the corresponding boundary cycle.
Since $H$ is $2$-colorable and since no vertices of $S(H)$ which are not the vertices of $H$ are adjacent, $S(H)$ is a balanced triangulation.
Conversely, for any balanced triangulation $G$ of $F^2$,
we can obtain an even embedding $H$ of $F^2$ such that $G=S(H)$ by removing vertices of one color from $G$.
We denote by $e(G)$ the number of edges of a graph $G$.
Since $|V(S(H))|$ equals the sum of $|V(H)|$ and the number of faces of $H$,
by Euler's formula,
for even embeddings $K$ and $K'$ of a fixed closed surface $F^2$
one has $|V(S(K))|>|V(S(K'))|$ if and only if $e(K)>e(K')$.

For each closed surface $F^2$,
there are infinitely many even embeddings whose minimal degree is at least $3$.
Hence the next result proves Theorem \ref{thm1}.

\begin{theorem}
\label{3.1}
Let $H$ and $K$ be even embeddings of a closed surface $F^2$ whose minimal degree is at least $3$.
If $S(H)$ is not isomorphic to $S(K)$, then $S(K)$ cannot be obtained from $S(H)$ by a sequence of balanced subdivisions and welds.
\end{theorem}

\begin{proof}
We may assume $e(H) \geq e(K)$,
and in particular $|V(S(H))| \geq |V(S(K))|$.
Let $G=S(H)$ and let $G' \ne G$ be a balanced triangulation of $F^2$ which can be obtained from $G$ by a sequence of balanced subdivisions and welds.
To prove the desired statement,
it is enough to prove that $|V(G)|<|V(G')|$.

\begin{figure}[htb]
\begin{center}
\input{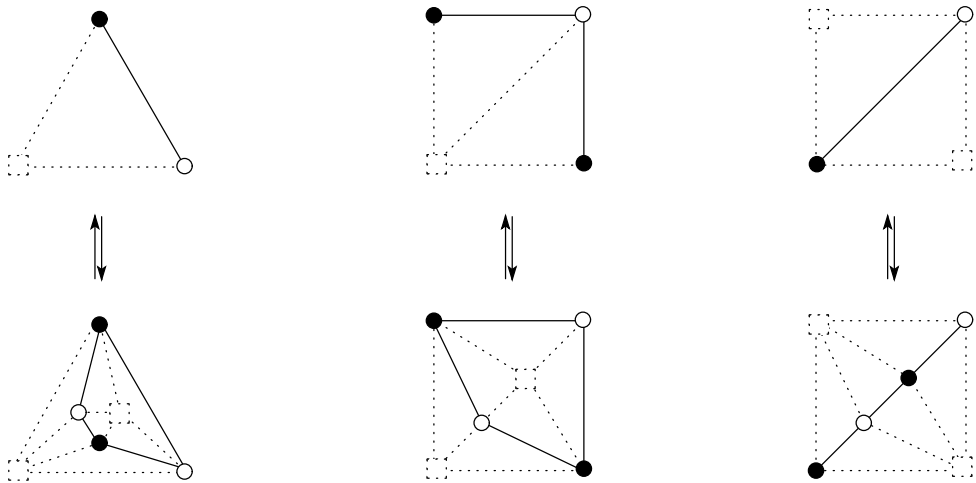}
\end{center}
\caption{Corresponding operations in $H$.}
\label{fig:HRB}
\end{figure}

Since balanced subdivision and welds preserve the balancedness,
there is an even embedding $H'$ of $F^2$ such that $G'=S(H')$ and is obtained from $H$ by a sequence 
of operations shown in Figure~\ref{fig:HRB}
which comes from balanced subdivisions and welds.
Furthermore, it is not difficult to check that 
each operation in Figure~\ref{fig:HRB} can be 
realized by a combination of the operations (I), (II), (III), (I'), (II') and (III').
Then, by Lemma \ref{lm:key},
the bipartite graph $H'$ is 
obtained from $H$ 
by a sequence of operations (I), (II) and (III).
Recall $G=S(H)$ and $G'=S(H')$.
Since $H=H'$ implies $S(H)=S(H')$ and since we assume $G \ne G'$, we have $e(H)<e(H')$,
which proves $|V(G)|<|V(G')|$ as desired.
\end{proof}

\begin{remark}
Face subdivisions $S(H)$ and $S(K)$ could be isomorphic even if $H \ne K$.
Indeed, for a balanced triangulation $G$,
one could obtain $3$ different even embeddings whose face subdivision is $G$ by removing the vertices of one color from $G$.
On the other hand, it is easy to make even embeddings $H$ and $K$ with $S(H)\ne S(K)$.
For example, if $e(H)\ne e(K)$, then we have $|V(S(H))| \ne |V(S(K))|$, and therefore $S(H)\ne S(K)$.
\end{remark}

\begin{remark}
The proof of Theorem \ref{3.1} says that, in the theorem, if we assume $e(H) \geq e(K)$, then
we do not need to assume that $K$ has minimal degree $\geq 3$.
For example, if $S(H)$ is the face subdivision of the cube and $S(K)$ is the octahedral sphere, then $S(K)$ cannot be obtained from $S(H)$ by a sequence of balanced subdivisions and welds.
\end{remark}

\section{Necessary operations for balanced triangulations}\label{sect:nece}
\label{sect:Necessary}

In this section, we discuss how many different types of cross-flips are necessary.
We first introduce operations called an $N$-{\em flip\/} and 
a $P_2$-{\em flip\/} originally defined in \cite{NSS}, as shown in 
Figure~\ref{fig:NandP}. 
(An $N$-flip is also found in cross-flips in \cite[Figure 1]{IKN}.)
Note that it is not allowed to make a double edge by the operations and each triangle in Figure~\ref{fig:NandP}
must be a face. 
Using those operations, Kawarabayashi et al.\ \cite{KNS} proved 
the following theorem. 

\begin{figure}[htb]
\begin{center}
\input{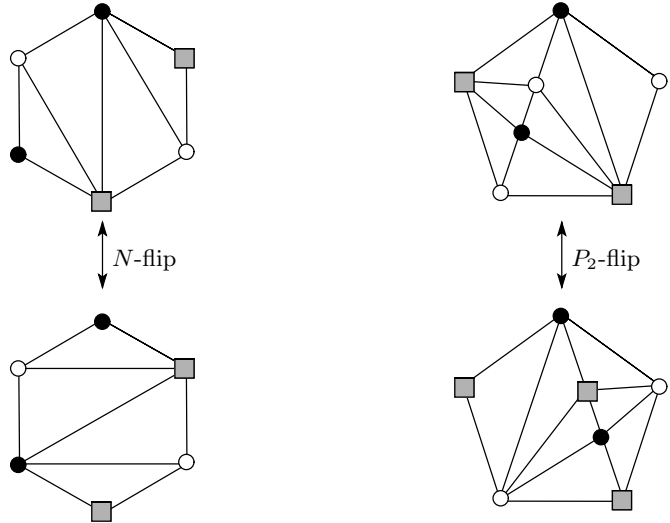}
\end{center}
\caption{$N$-flip and $P_2$-flip.}
\label{fig:NandP}
\end{figure}

\begin{theorem}[Kawarabayashi, Nakamoto and Suzuki \cite{KNS}]\label{thm:KNS}
For any closed surface $F^2$, there exists an integer $M$ such that any two 
balanced triangulations $G$ and $G'$ on $F^2$ with 
$|V(G)|=|V(G')|\geq M$ can be transformed into each other by a sequence 
of $N$- and $P_2$-flips.
\end{theorem}

We now prove Theorem \ref{thm:3operation} in the introduction, saying that BE-subdivisions, BE-welds and P-contractions are enough.

\begin{figure}[tb]
\begin{center}
\input{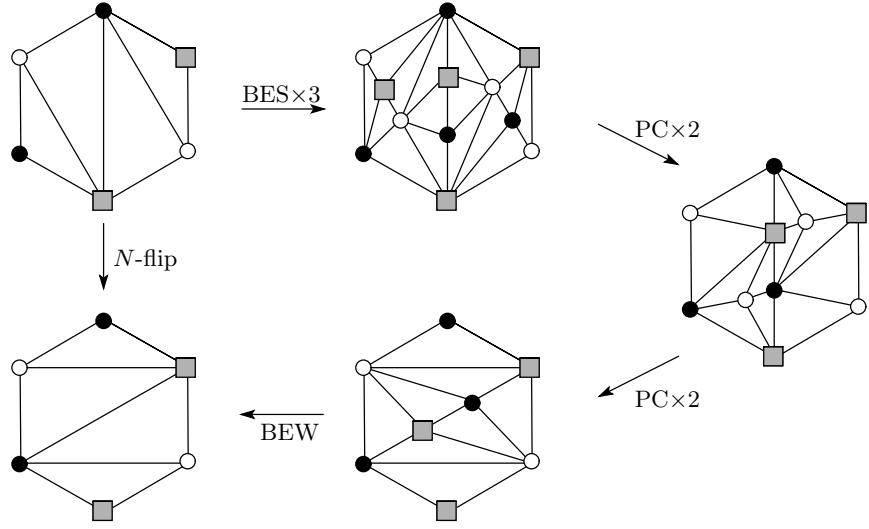}
\end{center}
\caption{An $N$-flip realized by a sequence of other operations.}
\label{fig:NN}
\end{figure}

\begin{proof}[Proof of Theorem~\ref{thm:3operation}]
Clearly, a $P_2$-flip can be replaced with a 
combination of a BE-subdivision and a BE-weld. 
Furthermore, an $N$-flip is replaced with a sequence of 
BE-subdivisions, P-contractions 
and a single BE-weld, as shown in Figure~\ref{fig:NN}. 
Since a BE-subdivision increases the number of the vertices by two and a P-contraction decreases the number of the vertices by one, the desired assertion follows from Theorem~\ref{thm:KNS}.
\end{proof}

Next, we show that most balanced triangulations of a fixed closed surface $F^2$ are connected by a sequence of P-contractions and P-splittings.
The following simple fact can be observed from Figure~\ref{fig:sph}.

\begin{lemma}
\label{lemX}
Let $G$ and $G'$ be balanced triangulations of a closed surface $F^2$ such that $G'$ is obtained from $G$ by applying the BE-subdivision to the edge $v_0v_1$ in $G$.
Let $xv_0v_1$ and $yv_0v_1$ be the faces of $G$ that contains $v_0v_1$ and let $u \ne v_0$ be the vertex such that  $xv_1u$ is a face of $G$.
If $uy$ is not an edge of $G$, then $G'$ is obtained from $G$ by a sequence of P-splittings.
\end{lemma}

\begin{figure}[tb]
\begin{center}
\input{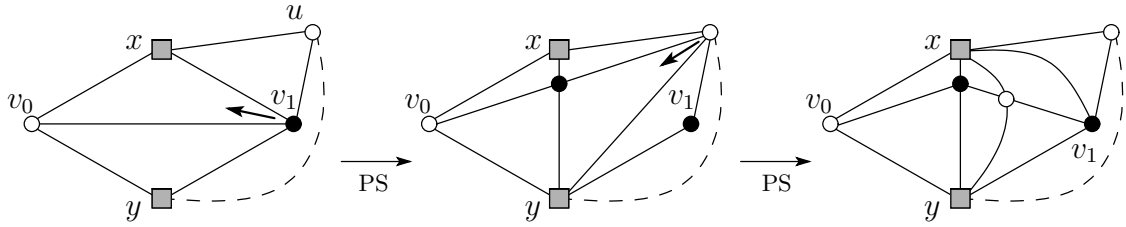}
\end{center}
\caption{Two PSs corresponding to a BES.}
\label{fig:sph}
\end{figure}

\begin{theorem} \label{finalthm}
Any two balanced triangulations of a closed surface $F^2$ 
other than finite 
exceptions (depending on $F^2$) can be 
transformed into each other by a sequence of P-splittings and 
P-contractions. 
\end{theorem}

\begin{proof}
First, observe that each of BE-subdivisions and BE-welds applied in 
Figure~\ref{fig:NN} 
satisfies the assumption of Lemma~\ref{lemX}.
Hence any $N$-flip can be replaced by a sequence of P-splittings and P-contractions.
Similarly, a $P_2$-flip is replaced with a combination of 
P-splittings and P-contractions by Lemma~\ref{lemX}. 
This observation also implies that if 
we can apply either an $N$-flip or a $P_2$-flip to a 
balanced triangulation, then we can apply a P-splitting.

Now, let $G$ and $G'$ be balanced triangulations of $F^2$ 
with $|V(G')| \geq |V(G)| \geq M$ where $M$ is the 
integer obtained in Theorem~\ref{thm:KNS}. 
By Theorem~\ref{thm:KNS}, $G$ can be transformed into 
another balanced triangulation with the same number of vertices 
by a sequence of $N$- and $P_2$-flips. 
Note that this implies that we can apply a P-splitting 
to $G$.
After applying a P-splitting to $G$, we obtain a balanced 
triangulation of $F^2$ with $|V(G)|+1$ vertices. 
We can repeat the argument until the number of vertices 
becomes $|V(G')|$; denote the resulting graph by $G_0$. 
By Theorem~\ref{thm:KNS} and the above argument 
$G_0$ and $G'$ can be transformed into each other by a sequence 
of P-splittings and P-contractions. 
Therefore, we conclude that $G$ and $G'$ are connected 
by only P-splittings and P-contractions. 
Then the assertion follows since there exist only finitely many balanced triangulations 
of $F^2$ with the number of vertices less than $M$.
\end{proof}

It would be natural to ask what are the exceptions in Theorem~\ref{finalthm}.
Let $F^2$ be a closed surface and let $M$ be an integer given in Theorem~\ref{thm:KNS}.
The proof of the Theorem~\ref{finalthm} says that two balanced triangulations are connected by a sequence of P-splittings and P-contractions if they have at least $M$ vertices.
We say that a balanced triangulation $G$ of $F^2$ is \textit{exceptional} if $G$ cannot be connected to a balanced triangulation $G'$ of $F^2$ with $|V(G')|\geq M$ by a sequence of $P$-splittings and P-contractions (this condition does not depend on a choice of $M$).
If we can apply a P-splitting to $G$, that is, there is a graph $G'$ such that $G'$ is obtained from $G$ by a P-splitting, then we can again apply a P-splitting to $G'$.
Thus if we can apply a P-splitting to $G$, then $G$ is not exceptional.
Also, if it is possible to apply a $P$-contraction to $G$, then it is also possible to apply a $P$-splitting to $G$.
Thus we have the following criterion.

\begin{proposition}
\label{thm:exceptional}
A balanced triangulation $G$ is not exceptional if and only if $G$ have faces $vwx,vxy,vyz$ such that 
$wz$ is not an edge of $G$.
\end{proposition}

We thinks that exceptional balanced triangulations are quite rare.
Indeed, for the $2$-sphere we have the following result, which proves Theorem \ref{thm:spheremain}.

\begin{theorem}\label{thm:spheremainlater}
The octahedral sphere is the only exceptional balanced triangulation of the $2$-sphere.
\end{theorem}

\begin{proof}
Let $G$ be an exceptional balanced triangulation of the $2$-sphere.
Since the octahedral sphere is the only triangulation of the $2$-sphere all whose vertices have degree $4$,
it suffices to show that every vertex of $G$ has degree $4$.

Let $v$ be a vertex of $G$ and $uv$ an edge of $G$.
We claim that $v$ has degree $4$.
Let $uvx$ and $uvy$ be the faces of $G$ that contains $uv$.
Also, let $z \ne u$ and $w \ne u$ be the vertices such that $vxz$ and $vwy$ are faces of $G$.
Note that $z \ne y$ since they have different colors, and similarly $w \ne x$.
By applying Lemma \ref{lemX} to faces $vxz,uvx$ and $uvy$,
we have that $yz$ must be an edge of $G$.
Similarly, by applying Lemma \ref{lemX} to faces $vwy, uvy,uvx$, we have that $xw$ must be an edge of $G$.
Then, since $G$ does not contains the complete bipartite graph of size $3$ by the planarity, $u$ must be equal to $w$, which implies that $v$ has degree $4$ as desired (see Figure~\ref{fig:sherenew}).
\end{proof}

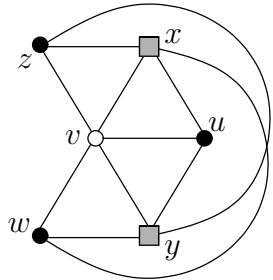
\begin{figure}[tb]
\begin{center}
{\unitlength 0.1in%
\begin{picture}( 13.7200, 14.3600)( 32.8000,-21.1300)%
\put(33.3000,-10.1000){\makebox(0,0)[lb]{$z$}}%
%
\special{sh 1.000}%
\special{ia 4310 1380 40 40  0.0000000  6.2831853}%
\special{pn 8}%
\special{ar 4310 1380 40 40  0.0000000  6.2831853}%
%
\special{sh 1.000}%
\special{ia 3450 890 40 40  0.0000000  6.2831853}%
\special{pn 8}%
\special{ar 3450 890 40 40  0.0000000  6.2831853}%
%
\special{sh 1.000}%
\special{ia 3450 1890 40 40  0.0000000  6.2831853}%
\special{pn 8}%
\special{ar 3450 1890 40 40  0.0000000  6.2831853}%
%
\special{pn 8}%
\special{pa 3450 1900}%
\special{pa 4020 1900}%
\special{fp}%
%
\special{pn 8}%
\special{pa 3450 900}%
\special{pa 4020 900}%
\special{fp}%
%
\special{pn 8}%
\special{pa 3450 900}%
\special{pa 3740 1380}%
\special{fp}%
%
\special{pn 8}%
\special{pa 3750 1400}%
\special{pa 4040 1880}%
\special{fp}%
%
\special{pn 8}%
\special{pa 4030 900}%
\special{pa 4320 1380}%
\special{fp}%
%
\special{pn 8}%
\special{pa 3450 1850}%
\special{pa 3740 1370}%
\special{fp}%
%
\special{pn 8}%
\special{pa 3750 1350}%
\special{pa 4040 870}%
\special{fp}%
%
\special{pn 8}%
\special{pa 4030 1860}%
\special{pa 4320 1380}%
\special{fp}%
%
\special{pn 8}%
\special{pa 3760 1380}%
\special{pa 4330 1380}%
\special{fp}%
%
\special{pn 8}%
\special{pa 3450 900}%
\special{pa 3478 883}%
\special{pa 3507 865}%
\special{pa 3563 831}%
\special{pa 3592 815}%
\special{pa 3620 799}%
\special{pa 3678 769}%
\special{pa 3707 755}%
\special{pa 3736 742}%
\special{pa 3766 729}%
\special{pa 3796 718}%
\special{pa 3826 708}%
\special{pa 3856 699}%
\special{pa 3887 692}%
\special{pa 3917 686}%
\special{pa 3949 681}%
\special{pa 3980 678}%
\special{pa 4012 677}%
\special{pa 4045 677}%
\special{pa 4077 679}%
\special{pa 4109 683}%
\special{pa 4141 688}%
\special{pa 4173 694}%
\special{pa 4205 702}%
\special{pa 4237 712}%
\special{pa 4268 723}%
\special{pa 4298 736}%
\special{pa 4327 750}%
\special{pa 4356 766}%
\special{pa 4384 783}%
\special{pa 4411 801}%
\special{pa 4437 821}%
\special{pa 4462 842}%
\special{pa 4485 865}%
\special{pa 4508 889}%
\special{pa 4529 913}%
\special{pa 4548 940}%
\special{pa 4566 967}%
\special{pa 4583 995}%
\special{pa 4598 1024}%
\special{pa 4611 1054}%
\special{pa 4622 1085}%
\special{pa 4632 1117}%
\special{pa 4640 1149}%
\special{pa 4646 1183}%
\special{pa 4650 1216}%
\special{pa 4652 1250}%
\special{pa 4653 1284}%
\special{pa 4651 1319}%
\special{pa 4648 1353}%
\special{pa 4644 1387}%
\special{pa 4637 1420}%
\special{pa 4629 1453}%
\special{pa 4619 1485}%
\special{pa 4608 1517}%
\special{pa 4595 1548}%
\special{pa 4581 1577}%
\special{pa 4565 1606}%
\special{pa 4547 1633}%
\special{pa 4528 1658}%
\special{pa 4508 1682}%
\special{pa 4486 1704}%
\special{pa 4462 1724}%
\special{pa 4438 1743}%
\special{pa 4412 1760}%
\special{pa 4385 1776}%
\special{pa 4357 1790}%
\special{pa 4328 1803}%
\special{pa 4298 1815}%
\special{pa 4268 1826}%
\special{pa 4237 1837}%
\special{pa 4173 1855}%
\special{pa 4140 1864}%
\special{pa 4041 1888}%
\special{pa 4030 1890}%
\special{fp}%
%
\special{pn 8}%
\special{pa 3440 1890}%
\special{pa 3468 1907}%
\special{pa 3497 1925}%
\special{pa 3553 1959}%
\special{pa 3582 1975}%
\special{pa 3610 1991}%
\special{pa 3668 2021}%
\special{pa 3697 2035}%
\special{pa 3726 2048}%
\special{pa 3756 2061}%
\special{pa 3786 2072}%
\special{pa 3816 2082}%
\special{pa 3846 2091}%
\special{pa 3877 2098}%
\special{pa 3907 2104}%
\special{pa 3939 2109}%
\special{pa 3970 2112}%
\special{pa 4002 2113}%
\special{pa 4035 2113}%
\special{pa 4067 2111}%
\special{pa 4099 2107}%
\special{pa 4131 2102}%
\special{pa 4163 2096}%
\special{pa 4195 2088}%
\special{pa 4227 2078}%
\special{pa 4258 2067}%
\special{pa 4288 2054}%
\special{pa 4317 2040}%
\special{pa 4346 2024}%
\special{pa 4374 2007}%
\special{pa 4401 1989}%
\special{pa 4427 1969}%
\special{pa 4452 1948}%
\special{pa 4475 1925}%
\special{pa 4498 1901}%
\special{pa 4519 1877}%
\special{pa 4538 1850}%
\special{pa 4556 1823}%
\special{pa 4573 1795}%
\special{pa 4588 1766}%
\special{pa 4601 1736}%
\special{pa 4612 1705}%
\special{pa 4622 1673}%
\special{pa 4630 1641}%
\special{pa 4636 1607}%
\special{pa 4640 1574}%
\special{pa 4642 1540}%
\special{pa 4643 1506}%
\special{pa 4641 1471}%
\special{pa 4638 1437}%
\special{pa 4634 1403}%
\special{pa 4627 1370}%
\special{pa 4619 1337}%
\special{pa 4609 1305}%
\special{pa 4598 1273}%
\special{pa 4585 1242}%
\special{pa 4571 1213}%
\special{pa 4555 1184}%
\special{pa 4537 1157}%
\special{pa 4518 1132}%
\special{pa 4498 1108}%
\special{pa 4476 1086}%
\special{pa 4452 1066}%
\special{pa 4428 1047}%
\special{pa 4402 1030}%
\special{pa 4375 1014}%
\special{pa 4347 1000}%
\special{pa 4318 987}%
\special{pa 4288 975}%
\special{pa 4258 964}%
\special{pa 4227 953}%
\special{pa 4163 935}%
\special{pa 4130 926}%
\special{pa 4031 902}%
\special{pa 4020 900}%
\special{fp}%
\put(35.8000,-14.1000){\makebox(0,0)[lb]{$v$}}%
\put(32.8000,-18.7000){\makebox(0,0)[lb]{$w$}}%
\put(43.3000,-13.4000){\makebox(0,0)[lb]{$u$}}%
\put(41.0000,-8.8000){\makebox(0,0)[lb]{$x$}}%
\put(41.0000,-20.2000){\makebox(0,0)[lb]{$y$}}%
%
\special{pn 0}%
\special{sh 0}%
\special{pa 3970 850}%
\special{pa 4070 850}%
\special{pa 4070 950}%
\special{pa 3970 950}%
\special{pa 3970 850}%
\special{ip}%
\special{pn 8}%
\special{pa 3970 850}%
\special{pa 4070 850}%
\special{pa 4070 950}%
\special{pa 3970 950}%
\special{pa 3970 850}%
\special{pa 4070 850}%
\special{fp}%
%
\special{pn 0}%
\special{sh 0.300}%
\special{pa 3970 850}%
\special{pa 4070 850}%
\special{pa 4070 950}%
\special{pa 3970 950}%
\special{pa 3970 850}%
\special{ip}%
\special{pn 8}%
\special{pa 3970 850}%
\special{pa 4070 850}%
\special{pa 4070 950}%
\special{pa 3970 950}%
\special{pa 3970 850}%
\special{pa 4070 850}%
\special{fp}%
%
\special{pn 0}%
\special{sh 0}%
\special{pa 3970 1840}%
\special{pa 4070 1840}%
\special{pa 4070 1940}%
\special{pa 3970 1940}%
\special{pa 3970 1840}%
\special{ip}%
\special{pn 8}%
\special{pa 3970 1840}%
\special{pa 4070 1840}%
\special{pa 4070 1940}%
\special{pa 3970 1940}%
\special{pa 3970 1840}%
\special{pa 4070 1840}%
\special{fp}%
%
\special{pn 0}%
\special{sh 0.300}%
\special{pa 3970 1840}%
\special{pa 4070 1840}%
\special{pa 4070 1940}%
\special{pa 3970 1940}%
\special{pa 3970 1840}%
\special{ip}%
\special{pn 8}%
\special{pa 3970 1840}%
\special{pa 4070 1840}%
\special{pa 4070 1940}%
\special{pa 3970 1940}%
\special{pa 3970 1840}%
\special{pa 4070 1840}%
\special{fp}%
%
\special{sh 0}%
\special{ia 3740 1380 40 40  0.0000000  6.2831853}%
\special{pn 8}%
\special{ar 3740 1380 40 40  0.0000000  6.2831853}%
\end{picture}}%
\end{center}
\caption{A configuration in the proof of Theorem~\ref{thm:spheremainlater}.}
\label{fig:sherenew}
\end{figure}

We close the paper with a few remarks and one question.

\begin{remark}
In Theorem~\ref{thm:KNS}, it is also true that there is a sequence of $N$-flips and $P_2$-flips
that transform $G$ into $G'$ and a given coloring of $G$ into a given coloring of $G'$ (this can be seen from the first paragraph of the proof of \cite[Theorem 3]{KNS}).
Thus, like \cite[Theorem 1.1]{IKN}, this stronger property is also true in Theorems~\ref{thm:3operation} and \ref{finalthm}.
\end{remark}

\begin{remark}
There is a balanced triangulation of the torus whose underlying graph is the complete tripartite graph $K_{3,3,3}$.
By Proposition \ref{thm:exceptional},
this triangulation is exceptional.
We do not know other examples of exceptional balanced triangulations.
\end{remark}

\begin{remark}
Any two balanced triangulations of a closed surface $F^2$ can be transformed into 
each other by a sequence of BT-subdivisions, BT-welds, P-contractions 
and P-splittings. 
Indeed, Figure \ref{fig:replace} shows that one can replace BE-subdivisions and BE-welds with combinations of BT-subdivisions, BT-welds, P-splittings and P-contractions.
\end{remark}

\begin{figure}[tb]
\begin{center}
\input{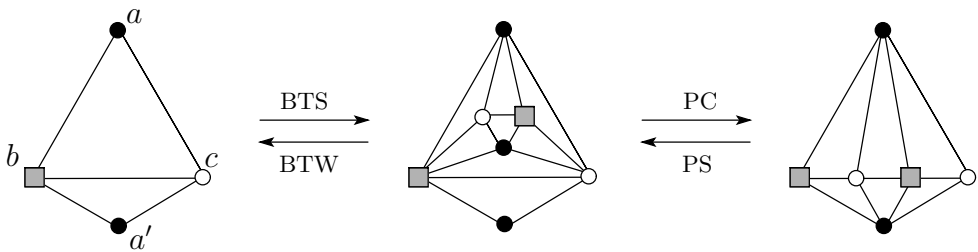}
\end{center}
\caption{Replacement of a BEW and a BES with the other four operations.}
\label{fig:replace}
\end{figure}

\begin{remark}
It was asked in \cite[Problem 4]{IKN} if two even triangulations of the same combinatorial manifold $M$ with the same coloring monodoromy are connected by cross-flips.
Since Theorem \ref{thm:KNS} also holds for even triangulations having the same monodoromy, the answer to this problem is yes for closed surfaces.
Also, Theorems~\ref{thm:3operation} and \ref{finalthm} hold in this generality.
\end{remark}

\begin{question}
Is there a generalization of Theorem~\ref{thm:3operation} (or Theorem~\ref{finalthm}) in higher dimension?
\end{question}

\end{document}